\documentclass[11pt]{article}
\usepackage[english]{babel}
\usepackage[width=15cm,height=20cm]{geometry}

\usepackage{tikz}
\usepackage{amsmath,amsfonts,amssymb,mathtools}
\usepackage{comment}
\usepackage{enumitem}
\usepackage{booktabs}%
\usepackage{amsthm}

\usepackage{pgfplots}
\pgfplotsset{compat=1.18}

\newtheorem{thm}{Theorem}
\newtheorem{prop}[thm]{Proposition}
\newtheorem{lem}[thm]{Lemma}
\newtheorem{cor}[thm]{Corollary}
\theoremstyle{definition}

\newtheorem{remark}[thm]{Remark}

\usepackage[colorlinks=true,allcolors=blue]{hyperref}

\newcommand{\pheq}{\phantom{=}\ }
\newcommand{\eqdef}{\coloneqq}
\newcommand{\al}{\alpha}

\newcommand{\ga}{\gamma}
\newcommand{\be}{\beta}
\newcommand{\de}{\delta}
\newcommand{\De}{\Delta}
\newcommand{\ka}{\varkappa}
\newcommand{\La}{\Lambda}
\newcommand{\la}{\lambda}
\renewcommand{\phi}{\varphi}
\newcommand{\tht}{\vartheta}
\newcommand{\fp}{\operatorname{fp}}
\newcommand{\asympt}{\operatorname{asympt}}
\newcommand{\gen}{\operatorname{gen}}
\newcommand{\ErAsympt}{\operatorname{E}^{\asympt}}

\newcommand{\ErMuAsympt}{\operatorname{G}^{\asympt}}

\newcommand{\bC}{\mathbb{C}}
\newcommand{\bD}{\mathbb{D}}
\newcommand{\bN}{\mathbb{N}}
\newcommand{\bR}{\mathbb{R}}

\newcommand{\bZ}{\mathbb{Z}}

\newcommand{\cB}{\mathcal{B}}

\newcommand{\cV}{\mathcal{V}}

\newcommand{\diag}{\operatorname{diag}}

\renewcommand{\Re}{\operatorname{Re}}
\renewcommand{\Im}{\operatorname{Im}}
\newcommand{\iu}{\operatorname{i}{}}
\newcommand{\enumber}{\operatorname{e}}
\newcommand{\medstrut}{\vphantom{\int_0^1}}
\newcommand{\hstrut}{\mbox{}\ \mbox{}}
\newcommand{\bigstrut}{\vphantom{\int_{0_0}^{1^1}}}

\title{Eigenvalues of one family of tridiagonal 
skew-self-adjoint Toeplitz matrices with complex perturbations on the corner}
\author{
C\'edric Bernardin, Sergei M. Grudsky,
\\
Egor A. Maximenko
and Alejandro Soto-Gonz\'alez\footnote{Corresponding author (email: \url{asoto@math.cinvestav.mx}).}
}

\begin{document}

\maketitle

\begin{center}
    \textit{
         Dedicated to the memory of our dear teacher, Igor Borisovich Simonenko.
    }
\end{center}

\begin{abstract}
In this paper,
we study the eigenvalues of the matrices
$T_n(a)+\gamma E_{n,1,1}$
where $T_n(a)$ is the Toeplitz matrix with generating symbol $a(t)=t-t^{-1}$,
$E_{n,1,1}$ is the $n\times n$ matrix whose upper left component is $1$ and the other components are zero,
and $\gamma$ is a fixed complex number such that $0<|\gamma|<1$.
As $n\to\infty$, the eigenvalues of these matrices are asymptotically distributed as the function $2 i \sin(x)$, $x\in[0,2\pi]$.
Our main result is an asymptotic formula for every eigenvalue with a residue of the order $O(1/n^3)$.


\end{abstract}

\section{Introduction}

Given a Laurent polynomial
\[
    b(t) = \sum_{j=-N}^M b_j t^j,
\]
we denote by $T_n(b)$ the Toeplitz matrix of order $n$ with generating symbol $b$:
\[
    T_n(b)
    \eqdef \bigl[ b_{j-k} \bigr]_{j,k=1}^n.
\]
We denote by $E_{n,1,1}$ the $n\times n$ matrix whose component $(1,1)$ is $1$ and the other components are $0$.
In terms of the Kronecker delta,
\[
E_{n,1,1}
\eqdef
\bigl[ \de_{j,1} \de_{k,1} \bigr]_{j,k=1}^n.
\]

In what follows,
let $a$ be the Laurent polynomial
$a(t)=t-t^{-1}$.
The corresponding Toeplitz matrices are tridiagonal.
They are not self-adjoint  
(i.e., not Hermitian)
but skew-self-adjoint.
For example,
\[
T_5(a)
=
\begin{bmatrix*}[r]
0 & -1 & 0 & 0 & 0 \\
1 & 0 & -1 & 0 & 0 \\
0 & 1 & 0 & -1 & 0 \\
0 & 0 & 1 & 0 & -1 \\
0 & 0 & 0 & 1 & 0
\end{bmatrix*}.
\]
The eigenvalues of $T_n(a)$ are well known, see~\eqref{eq:eigenvalues_tridiagonal_antisymmetric} in Section~\ref{sec:CharPol}.

In this paper,
we consider the family of $n\times n$
non-self-adjoint tridiagonal matrices $A_{\ga,n}$ defined by
\[
A_{\ga,n}
\eqdef T_n(a) + \ga E_{n,1,1}.
\]
For example,
\[
    A_{\ga,5}
    =
    \begin{bmatrix*}[r]
    \ga & -1 & 0 & 0 & 0 \\
    1 & 0 & -1 & 0 & 0 \\
    0 & 1 & 0 & -1 & 0 \\
    0 & 0 & 1 & 0 & -1 \\
    0 & 0 & 0 & 1 & 0
    \end{bmatrix*}.
\]
Our objective is to study the behavior of the eigenvalues of $A_{\ga,n}$ as $\ga$ is a fixed complex number such that $0<|\ga|<1$ and $n$ tends to infinity.
In a future paper, we are going present our study for values of $\ga$ such that $1\le |\ga|$.

Numerical experiments show that for $0<\ga<1$ and $n$ large enough,
the eigenvalues of $A_{\ga,n}$
concentrate near the segment
of the imaginary axis
joining the points $-2\iu$ and $2\iu$;
see
Figure~\ref{fig:eig}.

\begin{figure}[htb]
\centering
\includegraphics{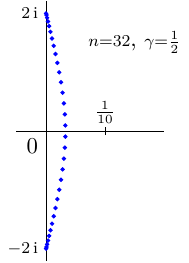}
\qquad\qquad
\includegraphics{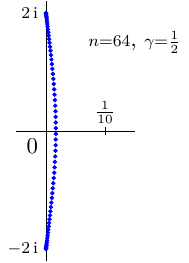}
\caption{Eigenvalues of $A_{1/2,32}$ (left)
and $A_{1/2,64}$ (right).
The X and Y scales are not equal.
\label{fig:eig}}
\end{figure}

Perturbed Toeplitz matrices appear in the discretizations of partial differential equations (see, for instance, LeVeque~\cite{Leveque2007} or Garoni and Serra-Capizzano~\cite{GS2017}) and in the study of resistor
networks (see Jiang, Zhou, Jiang, Zheng~\cite{JZJZ2023} and Zhang, Jiang, Zheng, Jiang~\cite{ZJZJ2024} and the references therein).
The matrices $A_{\ga,n}$ or very closely related matrices emerge in the problem of cut-off for harmonic chains of oscillators coupled with thermal baths (see, Eckmann and Hairer~\cite{EH2000},  Rieder, Lebowitz and Lieb~\cite{RLL1967}). 

For $\ga = \pm\iu$,  Yueh and Cheng~\cite{YuehCheng2008} computed the eigenvalues of $A_{\ga,n}$ explicitly.

Chorianopoulos and Famelis~\cite{CF2026} studied asymptotic expansions of the eigenvalues and eigenvectors of $A_{\ga,n}$ as $|\ga|\to0$ and $n$ is fixed.
They obtained asymptotic eigenvalue formulas with residue of order $O(|\ga|^5)$.
The methods of~\cite{CF2026} and the form of the asymptotic formulas are quite different from ours.
In Section~\ref{sec:NumTests}, we compare the precision of the obtained results. 

The study of Toeplitz matrices with finite rank perturbations is still under development.
B\"ottcher, Fukshansky, Garcia, and Maharaj~\cite{BFGM2014}, and Jia and Li~\cite{JL2015,JL2017}, provide calculus of the determinant for several types of Toeplitz matrices with perturbations in the four corners.

Da Fonseca, and Veerman~\cite{DV2009}, Yueh and Cheng ~\cite{YuehCheng2008}, and Zhang, 
Jiang and Jiang \cite{ZJJ2019}, analyzed the spectrum for some cases of tridiagonal Toeplitz matrices with perturbations in the four corners.
There the characteristic equation is presented, although asymptotic expansions are not computed.

Here, for arbitrary complex values of $\ga$ with $|\ga|<1$, we present eigenvalue asymptotic expansion of the eigenvalues of $A_{\ga,n}$.
So, our results extend previous works, moreover, the residue of our formulas are of order $O(1/n^3)$ as $n\to\infty$.

The procedure and results of this paper follow the structure of our previous works on this area~\cite{GMS2021,GMS2022,GMS2023,GMS2025}.
Those papers are inspired in the ``simple-loop method'', see~\cite{BGM2017,BBGM2015} and the references therein for a better complete introduction to this subject. 


\section{Main results}

Let $\bD$ be the open unit disk in the complex plane.
We treat $\ga\in\bC$ as a fixed parameter such that $\ga\in\bD\setminus\{0\}$, i.e., $0<|\ga|<1$.
We always suppose that $n$ is a natural number such that $n\ge 2$.

In Proposition~\ref{prop:CharPol}, we prove that the characteristic polynomial of $A_{\ga,n}$, which we denote by $D_{\ga,n}(\la)$, can be represented as the following linear combination of Chebyshev polynomials:
\[
   D_{\ga,n}(\la) \eqdef \det(\la I_n - A_{\ga, n}) = \iu^n
\left(U_n\left(-\frac{\iu\la}{2}\right)
+ \iu \ga U_{n-1}\left(-\frac{\iu\la}{2}\right)\right).
\]
Next, we apply the change of variable
$\la = \psi(z)$
where $\psi$ is the entire function
\begin{equation}
\label{eq:psi}
    \psi(z) 
    \eqdef
    \iu
    \bigl(
    \enumber^{\iu z}
    + \enumber^{-\iu z}
    \bigr),
    \qquad \text{that is,}\qquad
    \psi(z) = 2\iu\cos(z).
\end{equation}
Set
\begin{equation}
\label{eq:Delta_ga}
    \De_\ga\eqdef \ln\frac{1}{|\ga|}.
\end{equation}
We define $\tht_\ga \colon \bR + \iu (-\De_\ga, \De_\ga)\to\bC$ by
\begin{equation}
\label{eq:tht}
    \tht_\ga(z)
    \eqdef
    -\frac{\iu}{2}
    \ln\left( \frac{1+\iu\ga \enumber^{\iu z}}{1+\iu\ga \enumber^{-\iu z}}\right).
\end{equation}
It turns out (Proposition~\ref{prop:theta_analytic})
that the argument of the logarithm in~\eqref{eq:tht} does not take real negative values.
Therefore, $\tht_\ga$ is a well-defined analytic function.
Figure~\ref{fig:theta2} shows a plane-to-plane map of $\tht_\ga$.

\begin{figure}
    \centering   \includegraphics{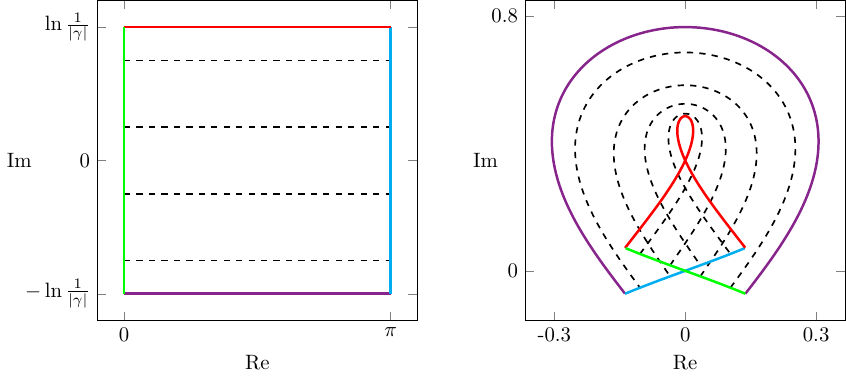}
    \caption{A graphical representation of $\tht_\ga$ for $\ga = 1/2$.}
    \label{fig:theta2}
\end{figure}

 \begin{thm}
 [main equation]
 \label{thm:MainEq}
        For every $n\ge 1$, and $z\in\bR + \iu (-\De_\ga, \De_\ga)$, the equality $D_{\ga,n}(\psi(z)) = 0$ holds if and only if there exists $j$ in $\bZ$ such that
        \begin{equation}
        \label{eq:main_eq}
            z = \frac{j\pi}{n+1} + \frac{\tht_\ga(z)}{n+1}
            .
        \end{equation}
    \end{thm}

Next, we will localize $j$ and $z$ more precisely.

For every $n$ in $\bN$, and every $j$ in $\{1,\ldots,n\}$, we set 
\begin{equation*}
    d_{n,j} \eqdef \frac{j\pi}{n+1},
    \qquad
    p_{\ga,n,j}\eqdef d_{n,j} + \frac{\tht_\ga(d_{n,j})}{n+1}.
\end{equation*}
In what follows, we use $p_{\ga,n,j}$ as the first approximation of the solution of~\eqref{eq:main_eq}, and $d_{n,j}$ can be seen as the zeroth approximation of the solution.

Now, we consider the following small disks around the points $p_{\ga,n,j}$:
\begin{equation}
\label{eq:B_j}
    \cB_{\ga,n,j}
    \eqdef
    \left\{
    z\in\bC \colon \left|z- p_{\ga,n,j}\right|<\frac{M_\ga}{(n+1)^2}
    \right\},
\end{equation}
where 
\begin{equation}
    \label{eq:M_gamma}
    M_\ga\eqdef 
    \frac{64\sqrt{|\ga|}}{(1-|\ga|)^3}.
\end{equation}
We provide a rough but explicit estimate $N_\ga$ for ``good'' values of $n$, such that the condition $n\ge N_\ga$ is sufficient for the following theorems:
\begin{equation}
\label{eq:N_gamma}
    N_\ga\eqdef 
    \frac{16}{(1-|\ga|)^2}.
\end{equation}

\begin{thm}[localization of the eigenvalues]
\label{thm:SolutionMainEq}
    For every $n\ge N_\ga$ and $j$ in $\{1,\ldots,n\}$, there exists a unique $s_{\ga,n,j}$ in $\cB_{\ga,n,j}$ that satisfies~\eqref{eq:main_eq}.   
    Moreover, $\psi(s_{\ga,n,j})$ is an eigenvalue of $A_{\ga,n}$.
\end{thm}

The proof of Theorem~\ref{thm:SolutionMainEq} lies in the fact that for $n\ge N_\ga$, the right-hand side of~\eqref{eq:main_eq} defines a contraction on $\cB_{\ga,n,j}$.
This implies also that the fixed iteration method can be applied to numerically compute the eigenvalues of $A_{\ga,n}$;
see Section~\ref{sec:NumTests} where we test this and later results.

\begin{cor}[separation of the eigenvalues]
\label{cor:distintc_eigvals}
    For every $n\ge N_\ga$, the numbers $\psi(s_{\ga,n,j})$ with $j$ in $\{1,\ldots,n\}$, are pairwise distinct eigenvalues of $A_{\ga,n}$.
    Moreover, the real parts of the numbers $s_{\ga,n,j}$ form a strictly increasing list contained in $(0,\pi)$:
    \begin{equation*}
        0< 
        \Re(s_{\ga,n,1}) < 
        \Re(s_{\ga,n,2}) <
        \cdots
        <
        \Re(s_{\ga,n,n})
        < 
        \pi.
    \end{equation*}
\end{cor}

Motivated by Corollary~\ref{cor:distintc_eigvals}, for every $n\ge N_\ga$ we index the eigenvalues of $A_{\ga,n}$ by
\[
    \la_{\ga,n,j} \eqdef \psi(s_{\ga,n,j})
    \qquad
    (1 \le j \le n).
\]

Next, we present the asymptotic formulas for the eigenvalues.
For this purpose, for every $n\in\bN$ we define
\begin{equation*}
    \La_{\ga,n}(z) \eqdef \psi(z) + \frac{\psi'(z)\tht_\ga(z)}{n+1} +\frac{\psi'(z)\tht_\ga(z)\tht_\ga'(z) + \frac{1}{2}\psi''(z)\tht_\ga(z)^2 }{(n+1)^2}.
\end{equation*}
For every $j$ in $\{1,\ldots,n\}$ set $\la_{\ga,n,j}^{\asympt}$ by
\begin{equation}
\label{eq:laj_Asympt}
    \la_{\ga,n,j}^{\asympt} \eqdef \La_{\ga,n}(d_{n,j}).
\end{equation}

\begin{thm}[asymptotic formula for the eigenvalues]
\label{thm:laj_expansion}
The following asymptotic approximation holds as $n\to\infty$, uniformly in $j$:
\begin{equation}
\label{eq:eigenvalue_approximation}
\la_{\ga,n,j}
= \la_{\ga,n,j}^{\asympt}
+ O\left(\frac{1}{n^3}\right).
\end{equation}
\end{thm}

Next theorem improves the asymptotic expansions of the eigenvalues that approaches $-2i$ and $2i$ as $n\to\infty$.
Hence, for every $n$ in $\bN$, and every $j$ in $\{1\ldots,n\}$, set
\begin{align}
    \label{eq:laExt_0}
    \la^{\operatorname{asympt},0}_{\ga,n,j} 
    &\eqdef 
    2\iu-\frac{\iu\pi^2 j^2}{(n+1)^2} + 
        \frac{2\ga}{1+\iu\ga} \frac{\pi^2 j^2}{(n+1)^3},
    \\
    \label{eq:laExt_1}
    \la^{\operatorname{asympt},1}_{\ga,n,j} 
     & \eqdef 
    -2\iu +
        \iu\pi^2\left(1-\frac{ j}{n+1}\right)^2 + 
        \frac{2\ga\pi^2}{1-\iu\ga} \frac{\left(1-\frac{ j}{n+1}\right)^2}{n+1}.
\end{align}

\begin{thm}[asymptotic formulas for the extreme eigenvalues]
\label{thm:ExtremeEigvals}
    The following asymptotic formulas hold:
    \begin{align}
    \label{eq:laj_ExtremeExpansionAt_0}
        \la_{\ga,n,j} & =
        \la^{\operatorname{asympt},0}_{\ga,n,j}  + O\left(\frac{j^4}{n^4}\right) 
        \qquad 
         \left(\frac{j}{n}\to0\right),
    \\
    \label{eq:laj_ExtremeExpansionAt_pi}
        \la_{\ga,n,j} &=
        \la^{\operatorname{asympt},1}_{\ga,n,j}  + O\left(\left(1-\frac{ j}{n+1}\right)^4\right)
        \qquad
        \left(\frac{j}{n}\to1\right).
    \end{align}
\end{thm}

Finally, we give a formula for the eigenvectors and their norms, in terms of $s_{\ga,n,j}$.

\begin{thm}[eigenvectors]
\label{thm:eigvec}
    For every $n\ge N_\ga$ and $j$ in $\{1,\ldots, n\}$, and every $k$ in $\{1,\ldots, n\}$, we define
    \begin{equation}
    \label{eq:eivec_trig_formula}
        v_{\ga,n,j,k}
        \eqdef 
        (-\iu)^{k} 
        \bigl(
        \sin(k s_{\ga,n,j})
        + \iu\ga \sin((k-1)s_{\ga,n,j})
        \bigr).
    \end{equation}
    Then the vector $v_{\ga,n,j}=[v_{\ga,n,j,k}]_{k=1}^n$ with components~\eqref{eq:eivec_trig_formula}
    is an eigenvector of $A_{\ga,n}$ associated to $\la_{\al,n,j}$.
    \end{thm}

For every $n\ge N_\ga$ and
$j$ in $\{1,\ldots,n\}$, set
\begin{equation}
\label{eq:mu_nu_xi}
\begin{gathered}
    \mu_{\ga,n,j} \eqdef 1+\iu\ga\cos(s_{\ga,n,j}),
    \qquad
    \nu_{\ga,n,j}  \eqdef \iu\ga\sin(s_{\ga,n,j}),
    \\
    \xi_{\ga,n,j} \eqdef  (\iu - \ga \cos(s_{\ga,n,j})) \ \overline{\ga\sin(s_{\ga,n,j})}.
\end{gathered}
\end{equation}
    
    \begin{thm}[norms of eigenvectors]
    \label{thm:norm_eigvecs}
    For every $n\ge N_\ga$ and $j$ in $\{1,\ldots, n\}$,
    \begin{equation}\label{eq:norm_eigvec}
       \begin{aligned}
        \|v_{\ga,n,j}\|_2^2 & = 
        \frac{|\nu_{\ga,n,j}|^2 - |\mu_{\ga,n,j}|^2}{2} \frac{\sin(n\Re(s_{\ga,n,j})) \cos((n+1)\Re(s_{\ga,n,j}))}{\sin(\Re(s_{\ga,n,j}))} 
        \\
        & \pheq \quad 
        + \frac{|\nu_{\ga,n,j}|^2 + |\mu_{\ga,n,j}|^2}{2} \frac{\sinh(n
        \Im(s_{\ga,n,j})) \cosh((n+1)
        \Im(s_{\ga,n,j}))}{\sinh(
        \Im(s_{\ga,n,j})
        )}
        \\
        & \pheq \quad
        +  \Re(\xi_{\ga,n,j}) \frac{\sin(n\Re(s_{\ga,n,j})) \sin((n+1)\Re(s_{\ga,n,j}))}{\sin(\Re(s_{\ga,n,j}))}
        \\
        & \pheq \quad
        - \Im(\xi_{\ga,n,j}) \frac{\sinh(n\Im(s_{\ga,n,j})) \sinh((n+1)\Im(s_{\ga,n,j}))}{\sinh(\Im(s_{\ga,n,j}))}.
    \end{aligned}
    \end{equation}  
\end{thm}

The rest of the paper is structured as follows.
In Section~\ref{sec:CharPol}, we compute the characteristic polynomial and transform it into a exponential-type equation~\eqref{eq:exponentialMainEq}.
In Section~\ref{sec:theta_Properties}, we define $\tht_\ga$ as the logarithm of the right-hand side of~\eqref{eq:exponentialMainEq} and study its properties, moreover, we prove Theorem~\ref{thm:MainEq}.
In Section~\ref{sec:MainEq}, we analyze the main equation~\eqref{eq:main_eq} and prove Theorem~\ref{thm:SolutionMainEq}.
In Section~\ref{sec:AsymtpFormulas}, we derive the asymptotic formulas
for the eigenvalues, in particular, we prove Theorems~\ref{thm:laj_expansion} and~\ref{thm:ExtremeEigvals}.
In Section~\ref{sec:eigvec}, we provide formulas for the eigenvectors and their norms.
Finally, in Section~\ref{sec:NumTests} we show the results of numerical experiments and compare them to~\cite{CF2026}.


\section{Characteristic polynomials}
\label{sec:CharPol}

Let $U_n$ be the Chebyshev polynomial of the second kind of degree $n$.
The following elementary properties are explained, e.g., in
\cite[Subsections 1.2.2 and 1.4.2]{MasonHandscomb2003}.
The polynomial sequence $(U_n)_{n=0}^\infty$
can be defined by the initial conditions
$U_0(x)=1$, $U_1(x)=2x$, and the recurrence relation
\begin{equation}
\label{eq:U_recurrence}
U_n(x) = 2x U_{n-1}(x) - U_{n-2}(x).
\end{equation}
It is also uniquely determined by the following identity:
\begin{equation}
\label{eq:U_main_identity}
U_n\left(\frac{1}{2} (t+t^{-1}) \right)
= \frac{t^{n+1}-t^{-n-1}}{t-t^{-1}}.
\end{equation}
For $t=\enumber^{\iu z}$
with $z$ in $\bC\setminus(\pi\bZ)$,
the identity~\eqref{eq:U_recurrence}
yields the following trigonometric expression
(for $z$ in $\pi\bZ$, it can be understood in the limit sense):
\begin{equation}
\label{eq:U_trigonometric}
U_n(\cos(z))
=\frac{\sin((n+1)z)}{\sin(z)}.
\end{equation}
Let $b_0$ be the Laurent polynomial
$b_0(t)\eqdef t+t^{-1}$.
It is well known
\cite[Section 1.5, Exercise~6]{MasonHandscomb2003}
that $U_n(x)$ can be expressed through the characteristic polynomial of the Toeplitz matrix $T_n(b_0)$:
\begin{equation*}
U_n(x)
= \det (2x I_n - T_n(b_0))
= \det (T_n(2x - b_0)).
\end{equation*}
Therefore, it is easy to see that
the characteristic polynomial of any tridiagonal Toeplitz matrix can be written in terms of $U_n$.
In our case,
for the Laurent polynomial $a(t)=t-t^{-1}$,
the characteristic polynomial of $T_n(a)$ is
\begin{equation}
\label{eq:det_Toeplitz_tridiagonal_antisymmetric}
\det (\la I_n - T_n(a))
= \det (T_n(\la - a))
= \iu^n U_n\left(-\frac{\iu \la}{2}\right).
\end{equation}
With the change of variables $\la=\psi(z)$ where $\psi$ is defined by~\eqref{eq:psi},
\[
\det (2\iu\cos(z)I_n-T_n(a))
= \iu^n U_n(\cos(z))
= \iu^n \frac{\sin((n+1)z)}{\sin(z)}.
\]
The zeros of this trigonometric expression are $x=\frac{k\pi}{n+1}$.
Therefore, the eigenvalues of $T_n(a)$ are
\begin{equation}
\label{eq:eigenvalues_tridiagonal_antisymmetric}
2\iu \cos\frac{k\pi}{n+1}
\qquad (1\le k\le n).
\end{equation}

We denote by $D_{\ga,n}$ the characteristic polynomial of $A_{\ga,n}$:
\[
D_{\ga,n}(\la)
\eqdef \det(\la I_n - A_{\ga, n}),
\]
that is,
\[
D_{\ga,n}(\la)
= \det(\la I_n - T_n(a) - \ga E_{n,1,1})
= \det(T_n(\la - a) - \ga E_{n,1,1}).
\]
For the first values of $n$,
\begin{align*}
D_{\ga,1}(\la)
&=\la-\ga,
\\[1ex]
D_{\ga,2}(\la)
&=\la^2-\ga \la + 1,
\\[1ex]
D_{\ga,3}(\la)
&= \la^3 - \ga \la^2 + 2\la - \ga,
\\[1ex]
D_{\ga,4}(\la)
&= \la^4 - \ga \la^3 + 3 \la^2 - 2\ga\la + 1.
\end{align*}

\begin{prop}[the characteristic polynomial of $A_{\ga,n}$
in terms of Chebyshev polynomials]
\label{prop:CharPol}
For each $\ga$ in $\bC$,
$n$ in $\bN$,
and $\la$ in $\bC$,
\begin{equation}
\label{eq:charpol}
D_{\ga,n}(\la)
= 
\iu^n
\left(U_n\left(-\frac{\iu\la}{2}\right)
+ \iu \ga U_{n-1}\left(-\frac{\iu\la}{2}\right)\right).
\end{equation}
\end{prop}

\begin{proof}
First, we expand the determinant
of $\la I_n-A_{\ga,n}$ by cofactors
along the first row.
For example, for $n=5$,
\begin{align*}
D_{\ga,n}(\la)
&=
\det (\la I_n- A_{\ga,5})
=
\begin{vmatrix*}[c]
\la-\ga & 1 & 0 & 0 & 0 \\
-1 & \la & 1 & 0 & 0 \\
0 & -1 & \la & 1 & 0 \\
0 & 0 & -1 & \la & 1 \\
0 & 0 & 0 & -1 & \la
\end{vmatrix*}
\\[1ex]
&=
(\la-\ga)
\begin{vmatrix*}[c]
\la & 1 & 0 & 0\\
-1 & \la & 1 & 0 \\
0 & -1 & \la & 1 \\
0 & 0 & -1 & \la
\end{vmatrix*}
-
\begin{vmatrix*}[c]
-1 & 1 & 0 & 0 \\
0 & \la & 1 & 0 \\
0 & -1 & \la & 1 \\
0 & 0 & -1 & \la
\end{vmatrix*}.
\end{align*}
Next, in the second term of the sum,
we expand the determinant by cofactors along the first column.
Thereby, we get the characteristic polynomials of the matrices $T_4(a)$ and $T_3(a)$:
\[
D_{\ga,5}(\la)
=
(\la-\ga)\det(T_4(\la-a))
+\det(T_3(\la-a)).
\]
Similar computations can be done for every $n\ge 3$.
Applying~\eqref{eq:det_Toeplitz_tridiagonal_antisymmetric} we get
\begin{equation*}
D_{\ga,n}(\la)
= (\la - \ga)
\iu^{n-1}U_{n-1}\left(-\frac{\iu \la}{2}\right)
+ \iu^{n-2} U_{n-2}\left(-\frac{\iu \la}{2}\right).
\end{equation*}
Using the recurrence relation~\eqref{eq:U_recurrence},
we transform~\eqref{eq:charpol}
in the following manner:
\begin{align*}
D_{\ga,n}(\la)
&=
(\la - \ga)
\iu^{n-1}U_{n-1}\left(-\frac{\iu \la}{2}\right)
+ \iu^{n-2} U_{n-2}\left(-\frac{\iu \la}{2}\right).
\\[1ex]
&= \iu^{n} 2\left(-\frac{\iu\la}{2}\right)
U_{n-1}\left(-\frac{\iu\la}{2}\right)
- \iu^n U_{n-2}\left(-\frac{\iu\la}{2}\right)
- \ga \iu^{n-1} U_{n-1}\left(-\frac{\iu\la}{2}\right)
\\
&= \iu^n U_n\left(-\frac{\iu\la}{2}\right)
- (-\iu) \ga \iu^n U_{n-1}\left(-\frac{\iu\la}{2}\right)
\\
&= 
\iu^n
\left(U_n\left(-\frac{\iu\la}{2}\right)
+ \iu \ga  U_{n-1}\left(-\frac{\iu\la}{2}\right)\right).
\end{align*}
A direct computation
shows that the obtained formula~\eqref{eq:charpol}
is also valid for $n=1$ and $n=2$.
\end{proof}

\begin{prop}
\label{prop:plusminus_2_are_not_eigenvalues}
For $\ga$ in $\bD$
and $n$ in $\bN$,
the numbers $2\iu$ and $-2\iu$ are not eigenvalues of $A_{\ga,n}$.
\end{prop}

\begin{proof}
In fact, for $\la=2\iu$,
we get
$U_n(-\iu\la/2)=U_n(1)=n+1$ and
\[
D_{\ga,n}(2\iu)
=\iu^n \bigl(n+1 + \iu \ga n\bigr).
\]
For $\la=-2\iu$,
we get
$U_n(-\iu\la/2)
=U_n(-1) =(-1)^n (n+1)$ and
\[
D_{\ga,n}(-2\iu)
=\iu^n \bigl((-1)^n (n+1) + \iu \ga (-1)^{n-1} n\bigr)
= (-\iu)^n
\bigl(n+1 - \iu \ga n\bigr).
\]
In both cases, the assumption $|\ga|<1$
ensures that $D_{\ga,n}(\la)\ne0$.
\end{proof}

Next, to simplify the characteristic equation,
we make the change of variables $\la=\psi(z)$ where $\psi$ is defined by~\eqref{eq:psi}.
We admit complex values of $z$.

\begin{prop}
\label{prop:exponential_MainEq}
    For each $\ga$ in $\bD$
    and each $n$ in $\bN$,
    the characteristic equation
    $D_{\ga,n}(\psi(z)) = 0$
    is equivalent to
    \begin{equation}
    \label{eq:exponentialMainEq}
    \enumber^{2\iu (n+1)z}
    = \frac{1+\iu\ga \enumber^{\iu z}}%
    {1+\iu\ga \enumber^{-\iu z}}.
    \end{equation}
\end{prop}

\begin{proof}
    Thanks to Proposition~\ref{prop:plusminus_2_are_not_eigenvalues},
    we know that the numbers $2\iu$ and $2\iu$ are not eigenvalues.
    Therefore, we can suppose that $z$ belongs to $\bC\setminus(\pi \bZ)$, which implies that $\enumber^{\iu z}
        - \enumber^{-\iu z}\neq 0$ .
    We substitute $\la = \psi(z)$ in~\eqref{eq:charpol}
    and apply~\eqref{eq:U_main_identity}:
    \begin{align*}
        D_{\ga,n}(\psi(z))
        &=
        \iu^n
        \left(U_n\left(\frac{1}{2}(\enumber^{\iu z} + \enumber^{-\iu z})\right)
        + \iu \ga U_{n-1}\left(\frac{1}{2}(\enumber^{\iu z} + \enumber^{-\iu z})\right)\right)
        \\
        &=
        \frac{\iu^n}{\enumber^{\iu z}
        - \enumber^{-\iu{}z}} 
        (\enumber^{\iu (n+1)z} - \enumber^{-\iu (n+1)z} +\iu\ga (\enumber^{\iu nz} - \enumber^{-\iu nz})).
    \end{align*}
    After regrouping,
    \begin{equation*}
    D_{\ga,n}(\psi(z))
    = 
    \frac{\iu^n \enumber^{-\iu(n+1)z}}{\enumber^{\iu z} - \enumber^{-\iu z}}
    \Bigl(
    \enumber^{2\iu(n+1)z}
    (1+\iu\ga\enumber^{-\iu z})
    - (1+\iu\ga\enumber^{\iu z})
    \Bigr).
    \end{equation*}
    Therefore,
    $D_{\ga,n}(\psi(z)) = 0$ is equivalent to~\eqref{eq:exponentialMainEq}.
\end{proof}

\begin{remark}[trigonometric form of the characteristic polynomial]
Making the change of variable $\la=\psi(z)$ in \eqref{prop:CharPol},
then applying~\eqref{eq:U_trigonometric} and
\[
    \sin(nz)
    =\sin((n+1)z-z)
    =\sin((n+1)z)\cos(z) - \cos((n+1)z)\sin(z),
\]
we express the characteristic polynomial of $A_{\ga,n}$ in the following trigonometric forms:
\begin{align}
\label{eq:charpol_trig1}
D_{\ga,n}(\psi(z))
&=
\iu^n\;
\frac{\sin((n+1)z) + \iu\ga \sin(nz)}{\sin(z)}.
\\[1ex]
\label{eq:charpol_trig2}
&=
\iu^n\;
\frac{(1+\iu\ga\cos(z))\sin((n+1)z)
-\iu\ga\sin(z)\cos((n+1)z)}{\sin(z)}.
\end{align}
Using~\eqref{eq:charpol_trig2},
the characteristic equation can be transformed to the form~\eqref{eq:main_eq},
but $\tht_\ga$ would be expressed via the complex arctangent function.
Instead of this, we prefer to deal with the more familiar logarithmic function.
\end{remark}


\section{\texorpdfstring{Properties of $\tht_{\ga}$}{Properties of theta}}
\label{sec:theta_Properties}

Let $Q_\ga(z)$ be the right-hand side of~\eqref{eq:exponentialMainEq}:
\begin{equation*}
    Q_\ga(z)
    \eqdef
    \frac{1+\iu \ga \enumber^{\iu  z}}%
    {1+\iu \ga \enumber^{-\iu z}}.
\end{equation*}
For $\ga=0$, the function $Q_\ga$ is just the constant $1$, and we will exclude this trivial case.
In this section,
we fix $\ga$ in $\bD\setminus\{0\}$.

To solve~\eqref{eq:exponentialMainEq},
we are going to apply the logarithm to $Q_\ga(z)$.
In this paper, we always use the principal value of the natural logarithm and 
the principal value of the argument (angle),
and we denote these functions by $\ln$ and $\arg$, respectively.
Recall that
\[
\ln\colon \bC\setminus(-\infty,0]\to \bR+\iu(-\pi,\pi),
\qquad
\arg\colon \bC\setminus(-\infty,0]\to(-\pi,\pi),
\]
$\ln$ is an analytic function satisfying $\ln(1)=0$ and $\ln'(z)=1/z$, and $\arg(z)$ is the imaginary part of $\ln(z)$.

In the next lemma, we show that $Q_\ga(z)$,
for a certain values of $z$,
belongs to the domain of $\ln$.
Recall that $\De_\ga$ is defined by~\eqref{eq:Delta_ga}.

\begin{lem}
\label{lem:theta_definition}
    For every $z$ in the strip
    $\bR + \iu (-\De_\ga,\De_\ga)$,  $Q_\ga(z)\in \bC \setminus (-\infty,0]$.
\end{lem}

\begin{proof}
    Let $z = x+\iu y$ where $x\in[0,\pi]$ and $y\in (-\De_\ga,\De_\ga)$.
    Then 
    \[
        |\enumber^{\iu z}|
        = \enumber^{-y}
        \le \enumber^{\de_\ga}
        = \enumber^{\ln\frac{1}{|\ga|}}
        = \frac{1}{|\ga|}.
    \]
    Therefore,
    $|\iu \ga \enumber^{\iu z}|<1$ and
    $\Re(1+\iu \ga \enumber^{\iu  z})>0$. 
    Similarly, \
    $\Re(1+\iu \ga \enumber^{-\iu z})>0$.
    Hence,
    \begin{equation}
        \label{eq:ArgumentsOf_Q}
        -\frac{\pi}{2}
        <
        \arg(1+\iu\ga \enumber^{\iu z}) 
        <
        \frac{\pi}{2},
        \qquad
        -\frac{\pi}{2}
        <
        \arg(1+\iu\ga \enumber^{-\iu z}) 
        < \frac{\pi}{2},
    \end{equation}
    and the quotient $Q_\ga(z)$
    cannot be a real negative number.    
\end{proof}

Now, the definition of $\tht_\ga$ by~\eqref{eq:tht} makes sense,
and we are ready to state the main properties of $\tht_\ga$.

We denote by $\cV_\ga$ the closed strip 
\begin{equation}
    \label{eq:V_gamma}
    \cV_\ga \eqdef \bR + \iu \left[-\frac{\De_\ga}{2},\frac{\De_\ga}{2}\right].
\end{equation}

\begin{prop}
\label{prop:theta_analytic}
    $\tht_\ga$ is an analytic
    function on the strip
    $\bR + \iu (-\De_\ga,\De_\ga)$,
    and all their derivatives are bounded functions on $\cV_\ga$.
    The following formulas hold for $\tht_\ga$ and its first derivative:
    \begin{align}
    \label{eq:tht_as_dif_log}
    \tht_\ga(z)
    &=
    -\frac{\iu}{2}
    \left(
    \log\left(1+\iu\ga\enumber^{\iu z}\right)
    -
    \log\left(1+\iu\ga\enumber^{-\iu z}\right)
    \right),
    \\
    \label{eq:tht_derivative}
    \tht_\ga'(z)
    &=
    \frac{\iu\ga\cos(z)-\ga^2}{1+2\iu\ga\cos(z)-\ga^2}
    =
    \frac{\iu\ga\cos(z)-\ga^2}%
    {(1+\iu\ga\enumber^{\iu z})%
    (1+\iu\ga\enumber^{-\iu z})}.
    \end{align}
\end{prop}

\begin{proof}
    These formulas and properties follow immediately from Lemma~\ref{lem:theta_definition} and its proof.
    Notice that $\tht_\ga$ is $2 \pi $-periodic, and its behavior on $\cV_\ga$ is determined by its behavior on the compact rectangle $[0,\pi] + [-\frac{\De_\ga}{2},\frac{\De_\ga}{2}]$.
\end{proof}

Theorem~\ref{thm:MainEq} follows from Propositions~\ref{prop:exponential_MainEq} and~\ref{prop:theta_analytic} because the equation~\eqref{eq:exponentialMainEq} can be rewritten in the form
\[
    \enumber^{2\iu (n+1)z} = \enumber^{2\iu \tht_\ga(z)},
\]
which is equivalent to the existence of $j$ in $\bZ$ such that $(n+1)z = j\pi + \tht_\ga(z)$.

For $q=0,1,\ldots$, 
\[
    \|\tht_\ga^{(q)}\|_\infty 
    \eqdef 
    \sup_{z\in\cV_\ga} |\tht_\ga^{(q)}(z)|.
\]
The objective of the rest of this section is to 
provide simple upper bounds for $\|\tht_\ga\|_\infty$ and $\|\tht_\ga'\|_\infty$,
and to compute the real and imaginary part of $\tht_\ga$ for $x$ in $[0,\pi]$.

\begin{prop}
For every $x$ in $[0,\pi]$,
\begin{equation}
\label{eq:tht_real_gamma_complex}
\begin{aligned}
\tht_\ga(x)
&=
\frac{1}{2}
 \arctan\frac{\Re(\ga) \cos(x)  - \Im(\ga)\sin(x)}{1-\Re(\ga) \sin(x) - \Im(\ga) \cos(x)} 
 \\
 & 
 \pheq
-  
\frac{1}{2}
\arctan\frac{\Re(\ga) \cos(x)  + \Im(\ga)\sin(x)}{1+\Re(\ga) \sin(x) - \Im(\ga) \cos(x)}
\\
&\pheq  +\frac{\iu}{4}
\ln\frac{1+|\ga|^2 - 2\Im(\ga) \cos(x) + 2\Re(\ga)\sin(x)}{1+|\ga|^2 - 2\Im(\ga) \cos(x) - 2\Re(\ga)\sin(x)}.
\end{aligned}
\end{equation}
\end{prop}
\begin{proof}
Let $x\in[0,\pi]$.
Recall that
\[
    \tht_\ga(x)
    = -\frac{\iu}{2} \ln (Q_\ga(x))
    \qquad
    \text{where} 
    \qquad 
    Q_\ga(x)
    = \frac{1+\iu\ga \enumber^{\iu x}}%
    {1+\iu \ga \enumber^{-\iu x}}.
\]
Let us consider the numerator of $Q_\ga(x)$:
\begin{align*}
   1+\iu\ga \enumber^{\iu x}
   &=
   1 + \iu (\Re(\ga) + \iu \Im(\ga))
   (\cos(x) + \iu \sin(x))
   \\
   &= 1-\Re(\ga) \sin(x)- \Im(\ga) \cos(x)
   + \iu (\Re(\ga) \cos(x) - \Im(\ga)\sin(x)).
\end{align*}
Using this expression,
we easily compute the absolute value
and the argument of $1+\iu\ga\enumber^{\iu x}$:
\begin{align*}
    |1+\iu\ga \enumber^{\iu x}|^2
    &=
    (1+\iu\ga \enumber^{\iu x})
    (1-\iu\overline{\ga} \enumber^{-\iu x}) 
    = 
     1 + |\ga|^2
     + 2\Re(\iu\ga \enumber^{\iu x})
     \\
     &= 
     1 + |\ga|^2 - 2\Re(\ga) \sin(x) - 2\Im(\ga) \cos(x),
    \\[1.5ex]
    \arg(1+\iu\ga \enumber^{\iu x})
    &= 
    \arctan\frac{\Re(\ga) \cos(x)
    - \Im(\ga)\sin(x)}%
    {1-\Re(\ga) \sin(x) - \Im(\ga) \cos(x)}.
\end{align*}
Similar computations can be performed for
$1+\iu\ga \enumber^{-\iu x}$.
Due to~\eqref{eq:ArgumentsOf_Q},
\[
    \arg(Q_\ga(x))
    = \arg(1+\iu\ga \enumber^{\iu x})
    - \arg(1+\iu\ga \enumber^{-\iu x}).
\]
Finally,
$\tht_\ga(x) = \frac{1}{2} \arg(Q_\ga(x)) + \frac{\iu}{2}\ln(|Q_\ga(x)|^{-1})$,
and we get the conclusion.
\end{proof}

\begin{prop}
\label{prop:theta_estimate}
    For every $z$ in $\cV_\ga$,
    \[
        |\tht_\ga(z)| \le \frac{4}{1-|\ga|}.
    \]
\end{prop}

\begin{proof}
We write $z$
as $x+\iu y$ where $x\in\bR$ and
$|y|\le \frac{1}{2}\Delta_\ga$.
The latter assumption implies that
\begin{equation*}
    \enumber^y \le \frac{1}{\sqrt{|\ga|}}
    \qquad \text{and} \qquad
    \enumber^{-y} \le \frac{1}{\sqrt{|\ga|}}.
\end{equation*}    
It follows that
\[
    1 - \sqrt{|\ga|} \le
    |1+\iu\ga \enumber^{\iu z}| 
    \le
    1 + \sqrt{|\ga|},
\]
and 
\[
    \ln(1 - \sqrt{|\ga|}) 
    \le 
    \Re (\ln(1+\iu\ga \enumber^{\iu z}) )
    \le 
    \ln(1 + \sqrt{|\ga|}).
\]
Notice that $\ln(1+t) \le -\ln(1-t)$ for every $t\in(0,1)$ because $\ln(1+t) + \ln(1-t) = \ln(1-t^2)<0$.
Therefore, 
\[
    |\Re (\ln(1+\iu\ga \enumber^{\iu z}) )| 
    \le
    \max\{-\ln(1 - \sqrt{|\ga|}),\ \ln(1 + \sqrt{|\ga|})\} 
    =
    -\ln(1 - \sqrt{|\ga|}).
\]
Similarly
    $|\Re (\ln(1+\iu\ga \enumber^{-\iu z}) )| 
    \le
    -\ln(1 - \sqrt{|\ga|})$.
    Thus,
    \[
    |\Im(\tht_\ga(z))|
    \le 
    \frac{1}{2} |\Re (\ln(1+\iu\ga \enumber^{\iu z}) ) - \Re (\ln(1+\iu\ga \enumber^{-\iu z}) )| 
    \le
    -\ln(1-\sqrt{|\ga|}).
    \]
    Moreover, we know that $|\Re(\tht_\ga(z))| \le \pi/2$.
    Then, using the elementary inequality 
    \[
    -\ln(1-t) \le 1/(1-t)
    \qquad(0\le t \le 1),
    \]
    we obtain
    \[
        |\tht_\ga(z)| \le 
        \frac{\pi}{2} + \ln\frac{1}{1-\sqrt{|\ga|}}
        \le 
         2 + \frac{1}{1-\sqrt{|\ga|}} 
         \le 
         2 + \frac{2}{1-|\ga|} 
         \le
         \frac{4}{1-|\ga|}.
        \]
\end{proof}

\begin{prop}
\label{prop:complex_derivative_tht_estimate}
For every $z$ in
$\cV_\ga$,
\begin{equation*}
|\tht_\ga'(z)|
\le \frac{8\sqrt{|\ga|}}{(1-|\ga|)^2}.
\end{equation*}
\end{prop}

\begin{proof}
We write $z$
as $x+\iu y$ where $x\in\bR$ and
$|y|\le \frac{1}{2}\Delta_\ga$.
The latter assumption implies that
\[
\enumber^y \le \frac{1}{\sqrt{|\ga|}}
\qquad \text{and} \qquad
\enumber^{-y} \le \frac{1}{\sqrt{|\ga|}}.
\]
We are going to estimate $|\tht'(z)|$ from above using the second expression in~\eqref{eq:tht_derivative}.
It is known that
\[
|\cos(z)|
\le \frac{1}{2}
\left( \enumber^{y} + \enumber^{-y} \right)
= \cosh(y).
\]
So, the numerator in the formula for
$\tht_\ga'(z)$ can be bounded by
\[
|\iu \ga \cos(z) - \ga^2|
\le \sqrt{|\ga|}+|\ga|^2
\le 2\sqrt{|\ga|}.
\]
Next,
we estimate from below the denominator in the formula for $\tht_\ga'(z)$:
\begin{align*}
\left|
\left(1 + \iu \ga \enumber^{\iu z}\right)
\left(1 + \iu \ga \enumber^{-\iu z}\right)
\right|
&\ge
(1 - |\ga| \enumber^{-y})
(1 - |\ga| \enumber^{y})
\\[0.5ex]
&\ge
(1 - \sqrt{|\ga|})^2
= \frac{(1 - |\ga|)^2}{(1 + \sqrt{|\ga|})^2}
\ge \frac{(1-|\ga|)^2}{4}.
\end{align*}
Joining these estimates we obtain the result.
\end{proof}


\section{Main equation and its analysis}
\label{sec:MainEq}

In this section,
we fix $\ga$ in $\bD\setminus\{0\}$.
Recall that $\cB_{\ga,n,j}$ is defined by~\eqref{eq:B_j}.
For every $n\in\bN$,
every $j$ in $\bZ$
and every $z$ in
$\bR + \iu (-\De_\ga,\De_\ga)$,
we denote the right-hand side of the main equation~\eqref{eq:main_eq} by $f_{\ga,n,j}(z)$:
\[
    f_{\ga,n,j}(z)
    \eqdef \frac{j\pi}{n+1} + \frac{\tht_\ga(z)}{n+1}.
\]
We will prove that for every $n$ large enough and every $j\in\{1,\ldots,n\}$, $f_{\ga,n,j}$ is a contraction on $\cB_{\ga,n,j}$, its respectively fixed point, say $s_{\ga,n,j}$, will be the solution of~\eqref{eq:main_eq} and $\psi(s_{\ga,n,j})$ the corresponding eigenvalue of $A_{\ga,n}$.

The following \emph{mean value inequality}
for complex-valued or vector-valued functions
is well known;
see, e.g.,~\cite[Theorem~1.1.1]{Hormander2003}.

\begin{prop}
\label{prop:MV}
    If $g\colon[a,b]\to\bC$ is differentiable at every point in $[a,b]$,
    then for every $x,y$ in $[a,b]$,
    \[
        |g(x)-g(y)| \le \|g'\|_\infty |x-y|.
    \]
\end{prop}

Recall that $M_\ga$, $N_\ga$ and $\cV_\ga$ are defined by~\eqref{eq:M_gamma},~\eqref{eq:N_gamma} and~\eqref{eq:V_gamma}, respectively.

Denote by $M_{\ga,0}$ and $M_{\ga,1}$ the upper bounds from Propositions~\ref{prop:theta_estimate} and~\ref{prop:complex_derivative_tht_estimate}, respectively:
\begin{equation}
    \label{eq:M_def}
    M_{\ga,0} \eqdef \frac{4}{1-|\ga|},
    \qquad
    M_{\ga,1} \eqdef
    \frac{8\sqrt{|\ga|}}{(1-|\ga|)^2}.
\end{equation}
Comparing~\eqref{eq:M_def} with~\eqref{eq:M_gamma} we see that $M_\ga$ is just $2 M_{\ga,0} M_{\ga,1}$.

\begin{lem}
\label{lem:B_in_strip}
    For every $n \ge N_\ga$ and every $j$ in $\{1,\ldots,n\}$, the ball 
    $\cB_{\ga,n,j}$ is contained in the strip $\cV_\ga$.
\end{lem}

\begin{proof}
    Let $n \ge N_\ga$, $j$ in $\{1,\ldots,n\}$, and $z$ in $\cB_{\ga,n,j}$.
    Then,
    \begin{align*}
        |\Im(z)| 
        & 
        \le 
        |\Im(p_{\ga,n,j})| + \frac{M_\ga}{(n+1)^2} 
         = 
        \frac{|\Im(\tht_\ga(d_{n,j}))|}{n+1} + \frac{M_\ga}{(n+1)^2} 
        \\
        &
        \le 
        \frac{1}{n+1} \cdot \frac{4}{1-|\ga|} + \frac{1}{(n+1)^2} \cdot \frac{64\sqrt{|\ga|}}{(1-|\ga|)^3}
        \le 
        \frac{1-|\ga|}{2}.
    \end{align*}
    In the last step, we have used the well-known elementary inequality 
    \[
        1-t \le -\ln(t) \qquad (0<t\le 1).
        \qedhere
    \]
\end{proof}

\begin{thm}
\label{thm:fk_is_contraction}
  For every $n \ge N_\ga$ and every $j$ in $\{1,\ldots,n\}$, $f_{\ga,n,j}$ is a contraction on $\cB_{\ga,n,j}$.
\end{thm}

\begin{proof}
Suppose $n\ge N_\ga$ and $j$ in $\{1,\ldots,n\}$.
By Lemma~\ref{lem:B_in_strip}, the ball $\cB_{\ga,n,j}$ is contained in the strip $[0,\pi]+ \iu [-\frac{\De_\ga}{2}, \frac{\De_\ga}{2}]$.
Therefore, $\cB_{\ga,n,j}$ is contained in the domain of $\tht_\ga$ and, moreover, we can apply upper bounds from Propositions~\ref{prop:theta_estimate} and~\ref{prop:complex_derivative_tht_estimate}.

For every $z$ in $ \cB_{\ga,n,j}$, by the mean value inequality (Proposition~\ref{prop:MV}),
\[
    \left| f_{\ga,n,j}(z)-p_{\ga,n,j}\right| 
    =
    \left|\frac{\tht_\ga (z) - \tht_\ga(d_{n,j})}{n+1}\right|
    \le 
    \frac{\|\tht_\ga'\|_\infty}{n+1}|z-d_{n,j}|
    \le 
    \frac{M_{\ga,1}}{n+1} |z-d_{n,j}|.
\]
Next, taking into account the definition~\eqref{eq:B_j} of $\cB_{\ga,n,j}$ and applying the triangle inequality,
we get
\begin{align*}
    \left| f_{\ga,n,j}(z)-p_{\ga, n,j}\right| 
    & \le
    \frac{M_{\ga,1}}{n+1} (|z-e_{\ga,n,j}| + |e_{\ga,n,j} - d_{n,j}|)
    \\
    & \le
    \frac{M_{\ga,1}}{n+1} \left(\frac{2M_{\ga,0}M_{\ga,1}}{(n+1)^2} + \frac{M_{\ga,0}}{n+1}\right)
    \\
    &\le 
    \frac{2M_{\ga,0}M_{\ga,1}}{(n+1)^2} \left( \frac{M_{\ga,1}}{N_\ga+1} + \frac{1}{2}\right)
    \le
    \frac{M_\ga}{(n+1)^2}.
\end{align*}
So, $f_{\ga,n,j}(z)$ belongs to $\cB_{\ga,n,j}$.

Finally, for every $z$ in $ \cB_{\ga,n,j}$, we have that
\[
    |f_{\ga,n,j}'(z)| =
    \frac{|\tht_\ga'(z)|}{n+1} 
    \le 
    \frac{M_{\ga,1}}{N_\ga+1}
    \le \frac{1}{2}.
    \qedhere
\]
\end{proof}

\begin{proof}[Proof of Theorem~\ref{thm:SolutionMainEq}]
    Let $n\ge N_\ga$ and  $j$ in $\{1,\ldots,n\}$.
    From Theorem~\ref{thm:fk_is_contraction} and the Banach fixed point theorem, there exists a unique number $s_{\ga,n,j}$  in $\cB_{\ga,n,j}$ that satisfies~\eqref{eq:main_eq}.
    The number
    $\psi(s_{\ga,n,j})$ is an eigenvalue of $A_{\ga,n}$.
\end{proof}

   \begin{proof}[Proof of Corollary~\ref{cor:distintc_eigvals}]
     Let $n\ge N_\ga$.
     First, we prove that if $1\le j\le n$, then $0 < \Re(s_{\ga,n,j})<\pi$.
    Indeed, since the real part of $\tht_\ga$ takes values in $(-\pi/2,\pi/2)$,
    \[
        \Re(s_{\ga,n,j}) = \frac{j\pi}{n+1} + \frac{\Re(\tht_\ga(s_{\ga,n,j}))}{n+1} \le \frac{(j+\frac{1}{2})\pi}{n+1} < \pi,
    \]
    and analogously $\Re(s_{\ga,n,j})>0$.
     
     Second, let us show that if $j,k$ in $\{1,\ldots,n\}$ and $s_{\ga,n,j} = s_{\ga,n,k}$, then $j = k$.
     From the main equation~\eqref{eq:main_eq},
     \[
        0 = s_{\ga,n,k} - s_{\ga,n,j}
        =
        \frac{(k-j)\pi}{n+1} + \frac{\tht_\ga(s_{\ga,n,k}) - \tht_\ga(s_{\ga,n,j})}{n+1}
        = 
        \frac{(k-j)\pi}{n+1},
     \]
    and $j=k$.

    Third, assuming that $j<k$ we will prove that $\Re(s_{\ga,n,j}) < \Re(s_{\ga,n,k})$:
    \[
        \Re(s_{\ga,n,k}) - \Re(s_{\ga,n,j})
        =
        \frac{(k-j)\pi}{n+1} + \frac{\Re(\tht_\ga(s_{\ga,n,k})) - \Re(\tht_\ga(s_{\ga,n,j}))}{n+1}
        > 
        \frac{(k-j-1)\pi}{n+1}
        >0.
    \]
    
     Fourth, we are going to prove that if $\psi(s_{\ga,n,j}) = \psi(s_{\ga,n,k})$, then $j=k$.
     We consider the difference of these values and apply the formula for the difference of cosines:
     \begin{align*}
        \psi(s_{\ga,n,j}) - \psi(s_{\ga,n,k}) 
        &=
        2\iu (\cos(s_{\ga,n,j})-\cos(s_{\ga,n,k})) 
        \\
        &= 
        -4\iu \sin\frac{s_{\ga,n,j} - s_{\ga,n,k}}{2} \sin\frac{s_{\ga,n,j} + s_{\ga,n,k}}{2}.
     \end{align*}
     By the first part of the proof we know that
     \[
        \left|\Re\left(\frac{s_{\ga,n,j}-s_{\ga,n,k}}{2} \right)\right|
        \le 
        \frac{\pi}{2}, 
        \qquad
        0
        <
        \Re\left(\frac{s_{\ga,n,j}+s_{\ga,n,k}}{2}\right) 
        < 
        \pi.
     \]
     Recall that $\sin(z) = 0$ if and only if $z\in \pi \bZ$.
     Therefore, we conclude that $\psi(s_{\ga,n,j}) - \psi(s_{\ga,n,k}) = 0$ only if $s_{\ga,n,j}-s_{\ga,n,k} = 0$, which, by the second part of the proof, implies that $j=k$.
   \end{proof}


\section{Asymptotic formulas}
\label{sec:AsymtpFormulas}

Let $\ga\in\bD\setminus\{0\}$.
In this section we deal with finding asymptotic expansions of $s_{\ga,n,j}$ and $\la_{\ga,n,j}$ for every $n$ large enough and every $j\in\{1,\ldots,n\}$.
To construct these expansions, we notice that for $n$ and $j$ as before, Theorem~\ref{thm:SolutionMainEq} states that
\begin{equation}
\label{eq:sj_asympt_order1}
    s_{\ga,n,j} = d_{n,j} + \frac{\tht_\ga(d_{n,j})}{n+1} + O\left(\frac{1}{(n+1)^2}\right).
\end{equation}

\begin{thm}
\label{thm:sj_expansion}
    The following asymptotic formula holds as $n\to\infty$, uniformly in $j$:
    \begin{equation}
    \label{eq:sj_expansion}
        s_{\ga,n,j} = d_{n,j} + \frac{\tht_\ga(d_{n,j})}{n+1} +  \frac{\tht_\ga(d_{n,j})\tht_\ga'(d_{n,j})}{(n+1)^2} +O\left(\frac{1}{(n+1)^3}\right).
    \end{equation}
\end{thm}

\begin{proof}
    For $n$ large and $j=1,\ldots,n$, Theorem~\ref{thm:SolutionMainEq} implies that $s_{\ga,n,j}$ satisfies the main equation~\eqref{eq:main_eq}:
    \[
        s_{\ga,n,j} = d_{n,j}+\frac{\tht_\ga(s_{\ga,n,j})}{n+1}.
    \]
    We substitute~\eqref{eq:sj_asympt_order1} in the right-hand side of the main equation and expand $\tht_{\ga}$ by Taylor around $d_{n,j}$:
    \begin{align*}
        s_{\ga,n,j} 
        &=
        d_{n,j}+\frac{\tht_\ga\left( d_{n,j} + \frac{\tht_\ga(d_{n,j})}{n+1} + O\left(\frac{1}{(n+1)^2}\right)\right)}{n+1}
        \\
        &= 
        d_{n,j} + \frac{\tht_\ga(d_{n,j})}{n+1} 
        + 
        \frac{\tht_\ga'(d_{n,j})
        \left(
        \tht_\ga(d_{n,j}) + O\left(\frac{1}{(n+1)^2}\right)
        \right)}{(n+1)^2} 
        \\
        &\pheq 
        \qquad\qquad
        + O\left(
            \frac{\left(\tht_\ga(d_{n,j}) + O\left(\frac{1}{(n+1)^2}\right)\right)^2}{n+1}
        \right).
    \end{align*}
    By Lemma~\ref{lem:theta_definition}, $\tht_\ga$ and all its derivatives are uniformly bounded in $\cB_{\ga,n,j}$.
    Therefore, the last expansions simplifies into~\eqref{eq:sj_expansion}.
\end{proof}

We are ready to provide the asymptotic expansion for the eigenvalues.

\begin{proof}[Proof of Theorem~\ref{thm:laj_expansion}]
    Recall that $\la_{\ga,n,j}$ is defined as $\psi(s_{\ga,n,j})$.
     Express $s_{\ga,n,j}$ by~\eqref{eq:sj_expansion} and expand $\psi$ by Taylor around $d_{n,j}$:
     \begin{align*}
         \la_{\ga,n,j}
         &=
         \psi\left(
         d_{n,j} + \frac{\tht_\ga(d_{n,j})}{n+1} +  \frac{\tht_\ga(d_{n,j})\tht_\ga'(d_{n,j})}{(n+1)^2} +O\left(\frac{1}{n^3}\right)
         \right)
         \\
         &=
         \psi(d_{n,j}) 
         + \frac{\psi'(d_{n,j}) \tht_\ga(d_{n,j}) }{n+1}
         + \frac{\psi'(d_{n,j}) \tht_\ga(d_{n,j}) \tht_\ga'(d_{n,j}) }{(n+1)^2} 
         +  O\left(\frac{\psi'(d_{n,j})}{n^3}\right)
         \\
         &\pheq \qquad 
         +\frac{\psi''(d_{n,j})}{2} \left( \frac{\tht_\ga(d_{n,j})}{n+1} +  \frac{\tht_\ga(d_{n,j})\tht_\ga'(d_{n,j})}{(n+1)^2} +O\left(\frac{1}{n^3}\right) \right)^2
         \\
         &\pheq \qquad
         +O\left(\left( \frac{\tht_\ga(d_{n,j})}{n+1} +  \frac{\tht_\ga(d_{n,j})\tht_\ga'(d_{n,j})}{(n+1)^2} +O\left(\frac{1}{n^3}\right) \right)^3\right).
     \end{align*}
     The fact that both $\tht_\ga$ and $\psi$ and their derivatives are bounded simplify last expression into~\eqref{eq:eigenvalue_approximation}.
\end{proof}

Due to Theorem~\ref{thm:fk_is_contraction}
and the analytic properties of
$\tht_\ga$,
we can obtain a more precise
approximation for $s_{\ga,n,j}$ with more asymptotic terms and so of $\la_{\ga,n,j}$
(see, e.g.,~\cite{BBGM2018},
where the scheme is explained).

Recall that $\tht_\ga'$ is given by~\eqref{eq:tht_derivative}.
Now we expand $\tht_\ga$ by Taylor around $0$ and $\pi$, respectively:
\begin{equation}
\label{eq:tht_Expansion}
\begin{aligned}
    \tht_\ga(z) &= \frac{\iu\ga}{1+\iu\ga} z + O(z^2)
    &(z\to 0),
    \\
    \tht_\ga(z) &= -\frac{\iu\ga}{1-\iu\ga}(z-\pi) + O((z-\pi)^2)
    &(z\to \pi).
\end{aligned}
\end{equation}

\begin{prop}
    The following asymptotic formulas hold:
    \begin{align}
    \label{eq:s_ExtremExapnsionAt_0}
        s_{\ga,n,j} &= d_{n,j}+ \frac{\iu\ga}{1+\iu\ga}\frac{d_{n,j}}{n+1} + O\left(\frac{d_{n,j}^2}{n+1}\right)
        \qquad
         \left(\frac{j}{n}\to0\right),
    \\
    \label{eq:s_ExtremExapnsionAt_pi}
        s_{\ga,n,j} &= d_{n,j} + \frac{\iu\ga}{1-\iu\ga}\frac{ \pi - d_{n,j}}{n+1} + O\left(\frac{(\pi - d_{n,j})^2}{n+1}\right)
        \qquad
        \left(\frac{j}{n}\to1\right).
    \end{align}
\end{prop}
\begin{proof}
  Let $j/n\to0$.
  Set $\ka\eqdef\frac{\iu\ga}{1+\iu\ga}$.
  From~\eqref{eq:sj_asympt_order1} and~\eqref{eq:tht_Expansion}
  we derive
  \[
    s_{\ga,n,j} = d_{n,j} + \frac{\ka d_{n,j}}{n+1} + O\left(\frac{j^2}{n^3}\right) + O\left(\frac{1}{n^2}\right) = d_{n,j} + O\left(\frac{d_{n,j}}{n+1}\right).
  \]
  We use the same trick as in the proof of Theorem~\ref{thm:sj_expansion}, i.e.,  substitute the last expression in the right-hand side of the main equation and expand $\tht_{\ga}$ by Taylor, but now around $0$:
  \begin{align*}
      s_{\ga,n,j} & =
      d_{n,j} +
      \frac{\tht_\ga\left(d_{n,j} + O\left(\frac{d_{n,j}}{n+1}\right)\right)}{n+1}
      \\
      & = d_{n,j} + \frac{\ka d_{n,j} + O\left(\frac{d_{n,j}}{n+1}\right) + O\left(d_{n,j}^2\right)+ O\left(\frac{d_{n,j}^2}{n+1}\right) + O\left(\frac{d_{n,j}^2}{(n+1)^2}\right)}{n+1}.
  \end{align*}
  The last line simplifies into~\eqref{eq:s_ExtremExapnsionAt_0}.

  The proof of~\eqref{eq:s_ExtremExapnsionAt_pi} follows in a similar manner.
\end{proof}

In the following proof of Theorem~\ref{thm:ExtremeEigvals}, we will use the following Taylor expansions of $\psi$ around $0$ and $\pi$, respectively:
\begin{equation}
\label{eq:psi_Expansion}
\begin{aligned}
    \psi(z) &= 2\iu - \iu z^2 + O(z^4)
    &(z\to 0),
    \\
    \psi(z) &= -2\iu + \iu (z-\pi)^2 + O((z-\pi)^4)
    &(z\to \pi).
\end{aligned}
\end{equation}

\begin{proof}[Proof of Theorem~\ref{thm:ExtremeEigvals}]
    Let $j/n\to0$.
    Recall that $\la_{\ga,n,j} = \psi(s_{\ga,n,j})$.
    Then, by~\eqref{eq:s_ExtremExapnsionAt_0} and~\eqref{eq:psi_Expansion},
    \begin{align*}
        \la_{\ga,n,j} & = 2\iu - 
        \iu\left(
        d_{n,j}+ \frac{\iu\ga}{1+\iu\ga}\frac{d_{n,j}}{n+1} + O\left(\frac{d_{n,j}^2}{n}\right)
        \right)^2 + O(d_{j,n}^4)
        \\
        &= 
        2\iu - \iu d_{n,j}^2 + 
        \frac{2\ga}{1+\iu\ga}\frac{d_{n,j}^2}{n+1} + O(d_{j,n}^4).
    \end{align*}
    Notice that $O(d_{j,n}^4) = O\left(j^4/n^4\right)$.
    Then,~\eqref{eq:laj_ExtremeExpansionAt_0} is derived from last line.

    The proof of~\eqref{eq:laj_ExtremeExpansionAt_pi} follows similarly.
\end{proof}


\section{Eigenvectors}
\label{sec:eigvec}

We conclude our theoretical analysis by computing the eigenvectors of $A_{\ga,n}$ and their norms.
The content of this section is mainly derived from~\cite[Section 3]{GMS2022}, see also~\cite{Ferguson1980,FF2009,YuehCheng2008}.

We omit the the proof of next proposition since it is identical to the proof of~\cite[Proposition 8]{GMS2022}.
Here, we use the convention that $U_{-1}(t) \eqdef 0$.

\begin{prop}[eigenvectors of $A_{\ga,n}$]
Let $\la\in\bC\setminus\{-2i,\ 2i\}$ be an eigenvalue of $A_{\ga,n}$.
Then, the vector $v = [v_k]_{k=1}^n$ with components
\begin{equation}
\label{eq:eigvec_AB_neq_0}
    v_k \eqdef 
    (-\iu)^{k+2}\left( 
    U_{k-1}\left(-\iu\frac{\la}{2}\right) + \iu \ga  U_{k-2}\left(-\iu\frac{\la}{2}\right)
    \right) = (-\iu)^{k+2} D_{\ga,k-1}(\la),
 \end{equation}
is an eigenvector of $A_{\ga,n}$ associated to $\la$.
\end{prop}

  Formula~\eqref{eq:eivec_trig_formula}, from Theorem~\ref{thm:eigvec}, follows after substituting $\la_{\ga,n,j}=\psi(s_{\ga,n,j})$ in~\eqref{eq:eigvec_AB_neq_0} and then mutiplying by $\sin(s_{\ga,n,j})$.

For the proof of the following theorem we require the following trigonometric and hyperbolic identities
\begin{equation}
\label{eq:trig_sum_identities}
    \begin{gathered}
        \sum_{k=1}^n \cos(2ks)  
        = 
        \frac{\sin(ns) \cos((n+1)s)}{\sin(s)},
        \qquad
        \sum_{k=1}^n \sin(2ks)  
        =
        \frac{\sin(ns) \sin((n+1)s)}{\sin(s)},
        \\
        \sum_{k=1}^n \cosh(2ks)  
        = 
        \frac{\sinh(ns) \cosh((n+1)s)}{\sinh(s)},
        \qquad
        \sum_{k=1}^n \sinh(2ks)  
        =
        \frac{\sinh(ns) \sinh((n+1)s)}{\sinh(s)}.
    \end{gathered}
\end{equation}

For every $n\ge N_\ga$ and
$j$ in $\{1,\ldots,n\}$ recall that $\mu_{\ga,n,j}, \nu_{\ga,n,j}$ and $\xi_{\ga,n,j}$ are defined by~\eqref{eq:mu_nu_xi}.

\begin{proof}[Proof of Theorem~\ref{thm:norm_eigvecs}]
    Let $n\ge N_\ga$ and $j$ and $k$ in $\{1,\ldots,n\}$.

    We apply the well-known trigonometric identity $\sin(x + y ) = \sin(x)\cos(y) + \cos(x)\sin(y) $ to~\eqref{eq:eivec_trig_formula}:
    \begin{align*}
        v_{\ga,n,j,k} & =(-\iu)^{k} 
        \bigl(
        \sin(k s_{\ga,n,j})
        + \iu\ga (\sin(ks_{\ga,n,j}) \cos(s_{\ga,n,j}) -\cos(ks_{\ga,n,j}) \sin(s_{\ga,n,j}) )
        \bigr).
        \\
        &= 
        (-\iu)^{k}\bigl(
        (1+\iu\ga \cos(s_{\ga,n,j}))\sin(k s_{\ga,n,j}) - \iu\ga \sin(s_{\ga,n,j}) \cos(ks_{\ga,n,j}) 
        \bigr).
    \end{align*}
    Then,
    \begin{align*}
        |v_{\ga,n,j,k}|^2 & =|1+\iu\ga \cos(s_{\ga,n,j})|^2 |\sin(ks_{\ga,n,j})|^2 + |\ga\sin(s_{\ga,n,j})|^2 |\cos(ks_{\ga,n,j})|^2
        \\
        &\pheq 
        + 2 \Re\left(
        \xi_{\ga,n,j} \sin(k s_{\ga,n,j}) \ \overline{ \cos(ks_{\ga,n,j})}
        \right).
    \end{align*}
    We write $s_{\ga,n,j} = \al+\iu\be$ where $\al = \Re(s_{\ga,n,j})$ and $\be = \Im(s_{\ga,n,j})$.
    Then
    \begin{align*}
        |\sin(ks_{\ga,n,j})|^2 & = 
        \sin^2(k\al) + \sinh^2(k\be),
        \\
        |\cos(ks_{\ga,n,j})|^2 & = 
        \cos^2(k\al) + \sinh^2(k\be),
        \\
        \sin(k s_{\ga,n,j})\ \overline{\cos(ks_{\ga,n,j})} &= 
        \cos(k\al) \sin(k\al) + \iu \cosh(k\be)\sinh(k\be).
    \end{align*}
   Therefore,
    \begin{align*}
        |v_{\ga,n,j,k}|^2 & = \frac{|\nu_{\ga,n,j}|^2 - |\mu_{\ga,n,j}|^2}{2} \cos(2k\al) + \frac{|\nu_{\ga,n,j}|^2 + |\mu_{\ga,n,j}|^2}{2} \cosh(2k\al)
        \\
        & \pheq \qquad 
        + \Re(\xi_{\ga,n,j}) \sin(2k\al) - \Im(\xi_{\ga,n,j}) \sinh(2k\be).
    \end{align*}
    Now, apply the identities~\eqref{eq:trig_sum_identities} to the sum $\sum_{k=1}^n|v_{\ga,n,j,k}|^2 $ and get~\eqref{eq:norm_eigvec}.
\end{proof}


\section{Numerical tests}
\label{sec:NumTests}

We have performed some numerical tests in SageMath.
For $\ga$ in $\bD\setminus\{0\}$,
$n$ in $\bN$,
and $j$ in $\{1,\ldots,n\}$,
we introduce the following notation.
\begin{itemize}
\item
$s_{\ga,n,j}^{\fp}$
is the numerical approximation of the fixed point of $f_{\ga,n,j}$
(see Theorem~\ref{thm:fk_is_contraction}).
\item
$\la_{\ga,n,j}^{\fp}
\eqdef
\psi(s_{\ga,n,j}^{\fp}) = 2\iu \cos(s_{\ga,n,j}^{\fp})$.
\item
$\la_{\ga,n,j}^{\asympt}$
is defined by~\eqref{eq:laj_Asympt}.
\item 
$\la_{\ga,n,j}^{\asympt,0}$ and $\la_{\ga,n,j}^{\asympt,1}$
are defined by~\eqref{eq:laExt_0} and~\eqref{eq:laExt_1}, respectively.
\item
$\la_{\ga,n,j}^{\gen}$
are the approximations of the eigenvalues of $A_{\ga,n}$
obtained by a general eigensolver implemented in SageMath.
\item $v^{\text{fp}}_{\al,n,j}$ is the vector whose components are the numerical approximations of~\eqref{eq:eivec_trig_formula}:
\[
v_{\ga,n,j,k}^{\text{fp}}
        \eqdef 
        (-\iu)^{k} 
        \bigl(
        \sin(k s_{\ga,n,j}^{\text{fp}})
        + \iu\ga \sin((k-1)s_{\ga,n,j}^{\text{fp}})
        \bigr)
        \qquad
        (1\le k\le n)
\]
\end{itemize}
All these values are computed by using the multiprecision arithmetic with $3322$ binary digits,
i.e., with approximately $1000$ decimal digits.

We have constructed a large series of examples including all values
$\ga = r \enumber^{2\pi \iu\al}$ with $r$ in $\{p/10\mid p=1,\ldots,9\}$
and $\al$ in $\{q/10\mid q=0,1,\ldots9\}$,
and all values of $n$ in $\{2,\ldots, 256\}$.
Surprisingly for us, the fixed point iteration method works for every value under consideration.
In all these examples, we have obtained
\[
    \max_{1\le j\le n}
    \|A_{\ga,n}v^{\text{fp}}_{\ga,n,j}-\la_{\ga,n,j}^{\text{fp}}v^{\text{fp}}_{\al,n,j}\|_2
    <10^{-996},
    \qquad
    \max_{1\le j\le n}
    |\la_{\ga,n,j}^{\fp}-\la_{\ga,n,j}^{\gen}|
    < 10^{-906}.
\]
These inequalities mean that $\la_{\ga,n,j}^{\text{fp}}$ and $v^{\text{fp}}_{\al,n,j}$
are correct up to the used precision
(the matrix-vector operations augment
the error by a few digits),
and the general multiprecision eigensolver in SageMath is not so precise for this matrix family.

To test the asymptotic formulas, we set
\[
    \ErAsympt_{\ga,n}
    \eqdef
    \max_{1\le j\le n}
    |\la_{\ga,n,j}^{\asympt}
    -\la_{\al,n,j}^{\fp}|.
\]
Table~\ref{table:errors} shows the results of some numerical experiments.
In these experiments,
the ``normalized maximal errors''
$n^3 \ErAsympt_{\ga,n}$
are bounded by some small positive numbers that increases as $|\ga|\to 1$,
which agrees with~\eqref{eq:eigenvalue_approximation}.

\begin{table}[htb]
\caption{Values of $\ErAsympt_{\ga,n}$
and $n^3 \ErAsympt_{\ga,n}$
for some $\ga$ and $n$.
\label{table:errors}}
{\small
\[
    \begin{array}{|c|c|c|c|c|c|c|}
    \hline
    & \multicolumn{2}{ |c| } {\bigstrut\ga=1/2} &  \multicolumn{2}{ |c| }{\bigstrut\ga=\iu/3} &\multicolumn{2}{ |c| }{\bigstrut\ga=2/5 - \iu 5/6}
    \\\hline
    \bigstrut n &
    \ErAsympt_{\ga,n} &
    \hstrut{} n^3 \ErAsympt_{\ga,n}\hstrut{} &
    \ErAsympt_{\ga,n} &
    \hstrut{} n^3 \ErAsympt_{\ga,n}\hstrut{} &
    \ErAsympt_{\ga,n} &
    \hstrut{} n^3 \ErAsympt_{\ga,n}\hstrut{}
    \\\hline
    \medstrut 256 & 3.01\times10^{-8} & 0.504 & 3.00\times10^{-9} & 0.050 & 5.58\times10^{-6} & 93.6
    \\
    \medstrut 512 & 3.79\times10^{-9} & 0.508 & 3.77\times10^{-10} & 0.051 & 7.26\times10^{-7} & 97.4
    \\
    \medstrut 1024 & 4.75\times10^{-10} &  0.510 & 4.72\times10^{-11} &  0.051  & 9.25\times10^{-8} &  99.4
    \\
    \medstrut 2048 & 5.96\times10^{-11} & 0.512 & 5.91\times10^{-12} & 0.051 & 1.17\times10^{-8} & 100.4
    \\
    \medstrut 4096 & 7.45\times10^{-12} & 0.512 & 7.39\times10^{-13} & 0.051  & 1.47\times10^{-9} & 100.9
    \\
    \hline
    \end{array}
\]
}
\end{table}

We test the asymptotic expansions of Theorem~\ref{thm:ExtremeEigvals}, for this purpose, set
\[
    \operatorname{E}^{\asympt,0}_{\ga,n,j}
    \eqdef
    |\la_{\ga,n,j}^{\asympt,0}
    -\la_{\al,n,j}^{\fp}|,
    \qquad
    \operatorname{E}^{\asympt,1}_{\ga,n,j}
    \eqdef
    |\la_{\ga,n,j}^{\asympt,1}
    -\la_{\al,n,j}^{\fp}|.
\]
Table~\ref{table:errors_extreme} shows the results of some numerical experiments.
As expected, compared with~\eqref{eq:laj_Asympt}, \eqref{eq:laExt_0} and~\eqref{eq:laExt_1} are better for small and large values of $j$ in $\{1,\ldots,n\}$, respectively.

\begin{table}[htb]
\caption{Values of $\operatorname{E}^{\asympt,k}_{\ga,n,j}$ and the corresponding normalized error, 
for $k=0,1$ and some $\ga$, $n$ and $j$.
\label{table:errors_extreme}}
{\small
\[
    \begin{array}{|c|c|c|c|c|}
    \hline
    & \multicolumn{4}{ |c| } {\bigstrut\ga=\iu/3} 
    \\\hline
    & \multicolumn{2}{ |c| } {\bigstrut j=1}
    & \multicolumn{2}{ |c| } {\bigstrut j=n}
    \\\hline
    \bigstrut n &
    \operatorname{E}^{\asympt,0}_{\ga,n,j} &
    \hstrut{} \frac{n^4}{j^4\vphantom{\int_M}}\operatorname{E}^{\asympt,0}_{\ga,n,j}\hstrut{} &
    \operatorname{E}^{\asympt,1}_{\ga,n,j} &
    \hstrut{}
    \frac{n^4}{(n+1-j)^4} \operatorname{E}^{\asympt,1}_{\ga,n,j}\hstrut{}
    \\\hline
    \medstrut 256 
    & 1.11\times10^{-10} & 0.484 
    & 1.45\times10^{-9} & 6.31 
    \\
    \medstrut 512 
    & 8.65\times10^{-12} & 0.599 
    & 9.08\times10^{-11} & 6.29 
    \\
    \medstrut 1024 
    & 5.95\times10^{-13} &  0.657 
    & 5.69\times10^{-12} &  6.28  
    \\
    \medstrut 2048 
    & 3.89\times10^{-14} & 0.686 
    & 3.56\times10^{-13} & 6.27 
    \\
    \medstrut 4096 
    & 2.49\times10^{-15} & 0.701 
    & 2.23\times10^{-14} & 6.27  
    \\
    \hline
    \end{array}
\]
}
\end{table}
We have verified numerically (for many values of parameters) the representations~\eqref{eq:charpol}, \eqref{eq:charpol_trig1}, \eqref{eq:charpol_trig2} for the characteristic polynomial,
the equivalence of the formula~\eqref{eq:tht_real_gamma_complex} for $\tht_\ga$,
expressions~\eqref{eq:eivec_trig_formula}, \eqref{eq:norm_eigvec} for the eigenvectors and their norms, respectively, 
and some other exact formulas of this paper.

For every $\ga\in\bD$ and every $n\ge 1$, define
\[
B_{\ga,n} \eqdef -\iu J_n A_{\ga,n} J_n^{-1}
\]
where $J_n\eqdef \diag(\iu, \iu^2, \ldots, \iu^n)$.
Easy computations prove that $\la$ is an eigenvalue of $A_{\ga,n}$
iff $\iu\la$ is an eigenvalue of $B_{\ga,n}$.

Chorianopoulos and Famelis~\cite[Corollary 1]{CF2026} considered the matrices $B_{\ga,n}$
supposing that $\ga\in\iu(-1,1)$.
They proved that as $n$ is fixed
and $\ga\to0$,
the eigenvalues of $B_{\ga,n}$
can be written as
\begin{equation}
\label{eq:mu_expansion}
    \mu_{\ga,n,j}=\mu_{\ga,n,j}^{\asympt}+O(|\ga|^5),
    \qquad
    \mu_{\ga,n,j}^{\asympt} \eqdef c_{j,0} + \iu \ga c_{j,1} -\ga^2 c_{j,2} -i\ga^3 c_{j,3} + \ga^4 c_{j,5}.
\end{equation}
Here $c_{j,0}$, $c_{j,1}$, etc.
are the following coefficients:
\begin{gather*}
    c_{j,0} \eqdef 2\cos(d_{n,j}), 
    \qquad
    c_{j,1} \eqdef \frac{2 \sin^2(d_{n,j})}{n+1},
    \qquad
    c_{j,2} \eqdef \frac{(2n-1) \sin^2(d_{n,j}) \cos(d_{n,j})}{(n+1)^2},
    \\
    c_{j,3} \eqdef \frac{2n(n-1)\sin^2(d_{n,j}) (4\cos^2(d_{n,j})-1)}{3(n+1)^3},
    \\
    c_{j,4} \eqdef \frac{(2n-3)(12n^2(2\cos^2(d_{n,j}) -1) +(2n+1)\sin^2(d_{n,j}) ) \sin^2(d_{n,j})\cos(d_{n,j}) }{12(n+1)^4},
\end{gather*}
and $d_{n,j}=\frac{j\pi}{n+1}$.

In order to compare the asymptotic expansions~\eqref{thm:laj_expansion} and~\eqref{eq:mu_expansion}, we set
\[
    \ErMuAsympt_{\ga,n}
    \eqdef
    \max_{1\le j\le n}
    |\mu_{\ga,n,j}^{\asympt}
    -\iu\la_{\al,n,j}^{\fp}|.
\]
Table~\ref{table:errors2} shows
these errors for some values of $n$ and $\ga$,
As expected, \eqref{eq:mu_expansion}
can be better for very small values of $\ga$
(say, $|\ga|=10^{-3}$),
and our asymptotic formulas can be better
for large $n$ and not so small values of $\ga$
(say, $|\ga|=10^{-2}$).

\begin{table}[htb]
\caption{ Values of $\ErAsympt_{\ga,n}$
and $\ErMuAsympt_{\ga,n}$
for some $\ga$ and $n$.
\label{table:errors2}}
{\small
\[
    \begin{array}{|c|c|c|c|c|c|c|}
    \hline
    &
    \multicolumn{2}{ |c| }{\bigstrut\ga=\iu /2}
    &
    \multicolumn{2}{ |c| }{\bigstrut\ga=-\iu /100}
    &
    \multicolumn{2}{ |c| }{\bigstrut\ga=\iu /1000}
    \\\hline
    \bigstrut n &
    \ErAsympt_{\ga,n} &
    \hstrut{} \ErMuAsympt_{\ga,n}\hstrut{} 
    & \ErAsympt_{\ga,n} &
    \hstrut{} \ErMuAsympt_{\ga,n}\hstrut{} 
    & \ErAsympt_{\ga,n} &
    \hstrut{} \ErMuAsympt_{\ga,n}\hstrut{} 
    \\\hline
    \medstrut 256 & 1.11\times10^{-8} & 4.51\times10^{-5} & 7.85\times10^{-14} & 1.52\times10^{-13} & 7.85 \times10^{-17} & 1.52\times10^{-18}
    \\
    \medstrut 512 & 1.40\times10^{-9} & 2.29\times10^{-5} & 9.88\times10^{-15} & 7.70\times10^{-14} & 9.88\times10^{-18} & 7.70\times10^{-19}
    \\
    \medstrut 1024 & 1.76\times10^{-10} &  1.15\times10^{-5} & 1.24\times10^{-15} &  3.88\times10^{-14} & 1.24\times10^{-18} & 3.88 \times10^{-19}
    \\
    \medstrut 2048 & 2.20\times10^{-11} & 5.79\times10^{-6}  & 1.55\times10^{-16} & 1.95\times10^{-14} & 1.55\times10^{-19} & 1.95\times10^{-19}
    \\
    \medstrut 4096 & 2.75\times10^{-12} & 2.90\times10^{-6} & 1.94\times10^{-17} & 9.75\times10^{-15} & 1.94\times10^{-20} & 9.75\times10^{-20}
    \\
    \hline
    \end{array}
\]
}
\end{table}

\subsection*{Declarations}

The first author has been supported by a research project HSE-BR-2025-039 implemented as part of the Basic Research Program at HSE University.
The second author was supported by Regional Mathematical Center of the Southern Federal University
with the support of the Ministry of Science and Higher Education of Russia, Agreement 075-02-2026-1316.
The second, third, and fourth authors have been supported by SECIHTI (Mexico) project ``Ciencia de Frontera'' FORDECYT-PRONACES/61517/2020.
The third author also has been supported by IPN-SIP project 20260520
(Instituto Polit\'ecnico Nacional, Mexico).

\medskip\noindent
The authors state that there is no conflict of interest. 

\medskip\noindent
The authors declare that this article has no associated data.

\section*{Authors' data}

C\'{e}dric Bernardin,
\\
sedric.bernardin@gmail.com,
\quad
https://orcid.org/0000-0002-3467-2804.
\\
National Research University Higher School of Economics,
Moscow, Russia.

\medskip\noindent
Sergei M. Grudsky,
\\
grudsky@math.cinvestav.mx,
\quad
https://orcid.org/0000-0002-3748-5449.
\\
CINVESTAV del IPN, Departamento de Matem\'aticas, Ciudad de M\'exico, Mexico.

\medskip\noindent
Egor A. Maximenko,
\\
emaximenko@ipn.mx,
\quad
https://orcid.org/0000-0002-1497-4338.
\\
Instituto Polit\'ecnico Nacional,
Escuela Superior de F\'isica y Matem\'aticas,
Ciudad de M\'exico, Mexico.

\medskip\noindent
Alejandro Soto-Gonz\'alez,
\\
asoto@math.cinvestav.mx,
\quad
https://orcid.org/0000-0003-2419-4754.
\\
CINVESTAV del IPN,
Departamento de Matem\'aticas,
Ciudad de M\'exico, Mexico.

\end{document}